\documentclass{amsart}
\pdfoutput=1

\usepackage{amssymb}
\usepackage{booktabs}
\usepackage{color}
\usepackage[shortlabels]{enumitem}
\usepackage{graphicx}
\usepackage{mathtools}
\usepackage{thmtools}
\usepackage{url}

\usepackage[utf8]{inputenc}
\usepackage[T1]{fontenc}

\usepackage[pagebackref]{hyperref} 
\usepackage[all]{hypcap} 

\definecolor{greenish}{rgb}{0.0, 0.7, 0.0}
\definecolor{purplish}{rgb}{0.5, 0.0, 0.8}

\hypersetup{
  breaklinks,colorlinks,citecolor=greenish,linkcolor=purplish, 
  pdftitle={A modular relation for the chromatic symmetric functions of (3+1)-free posets},
  pdfauthor={Mathieu Guay-Paquet},
}

\graphicspath{{graphics/}}

\newcommand{\PP}{\mathbb{P}}

\newcommand{\QQ}{\mathbb{Q}}

\newcommand{\abs}[1]{\left|{#1}\right|}

\newcommand{\shortmeq}{\mathrel{\sim_{\text{M}}}}
\newcommand{\shortpeq}{\mathrel{\sim_{\text{P}}}}
\newcommand{\tallmeq}{\mathrel{\underset{\text{M}}{\scalebox{2}[1]{$\sim$}}}}
\newcommand{\tallpeq}{\mathrel{\underset{\text{P}}{\scalebox{2}[1]{$\sim$}}}}
\newcommand{\meq}{\mathchoice{\tallmeq}{\shortmeq}{\shortmeq}{\shortmeq}}
\newcommand{\peq}{\mathchoice{\tallpeq}{\shortpeq}{\shortpeq}{\shortpeq}}

\newcommand{\graphic}[1]{\left(\vcenter{\hbox{\includegraphics[scale=.6, trim=-2pt -2pt -2pt -2pt]{#1}}}\right)}

\DeclareMathOperator{\csf}{CSF}

\declaretheorem[numberwithin=section]{theorem}

\declaretheorem[numberlike=theorem]{proposition}

\declaretheorem[numberlike=theorem, style=definition]{remark}
\declaretheorem[numberlike=theorem, style=definition]{example}

\numberwithin{equation}{section} 

\begin{document}

\title[A modular law for the poset CSF]{A modular law for the chromatic \\ symmetric functions of $(3+1)$-free posets}

\author{Mathieu Guay-Paquet}
\address{
LaCIM \\
Universit\'e du Qu\'ebec \`a Montr\'eal \\
201 Pr\'esident-Kennedy \\
Montr\'eal QC\ \ H2X~3Y7 \\
Canada}
\thanks{This research was supported by an NSERC Postdoctoral Fellowship} 
\email{mathieu.guaypaquet@lacim.ca}

\begin{abstract}
We consider a linear relation which expresses Stanley's chromatic symmetric function for a poset in terms of the chromatic symmetric functions   of some closely related posets, which we call the modular law. By applying this in the context of $(3+1)$-free posets, we are able to reduce Stanley and Stembridge's conjecture that the chromatic symmetric functions of all $(3+1)$-free posets are $e$-positive to the case of $(3+1)$-and-$(2+2)$-free posets, also known as unit interval orders. In fact, our reduction can be pushed further to a much smaller class of posets, for which we have no satisfying characterization. We also obtain a new proof of the fact that all 3-free posets have $e$-positive chromatic symmetric functions.
\end{abstract}

\maketitle

\section{Introduction}\label{sec:intro}

In~\cite{stanley95}, Stanley generalized the familiar notion of chromatic polynomials for (finite) graphs in two directions to obtain the notion of chromatic \emph{symmetric functions}, defined for either graphs or \emph{posets}. Instead of counting how many proper vertex colourings there are for a given number of colours, the chromatic symmetric function is a generating function for all proper vertex colourings which keeps track of the number of vertices in each colour class. Thus, for a graph $G$ with vertex set $V$, the chromatic symmetric function is
\begin{align*}
  \csf(G)
    &= \sum_{\mathclap{\substack{\text{proper} \\ \kappa \colon V \to \PP}}} \mathbf{x}_\kappa \\
    &= \sum_{\mathclap{\substack{\text{proper} \\ \kappa \colon V \to \PP}}} x_1^{\#\kappa^{-1}(1)} x_2^{\#\kappa^{-1}(2)} x_3^{\#\kappa^{-1}(3)} \cdots,
\end{align*}
where $\PP = \{1, 2, 3, \ldots\}$ and $\mathbf{x} = (x_1, x_2, x_3, \ldots)$ is a countable set of indeterminates. For graphs, a colouring is proper if each colour class is an independent set. For posets, a colouring is proper if each colour class is a chain. Equivalently, $\csf(P)$ for a poset $P$ is the same as $\csf(G)$ for the incomparability graph $G$ of $P$.

The set of colours $\PP$ can be freely permuted without affecting the definition of $\csf(P)$, so the chromatic symmetric function is in fact a symmetric function. We may ask, as Stanley did, about its expansion in the classical bases for the ring of symmetric functions. In the basis of monomial symmetric functions, the coefficient of $m_\lambda$ in $\csf(P)$ is simply the number of proper colourings of $P$ where $\lambda_i$ vertices have colour $i$ for each $i \in \PP$. In particular, all the coefficients in this basis are nonnegative, so we say that $\csf(P)$ is \emph{$m$-positive} for all posets $P$.

In the basis of power sum symmetric functions, the coefficient of $p_\lambda$ in $\csf(P)$ can be obtained a mobius inversion argument (see~\cite[Theorem~2.6]{stanley95}), and it is not positive in general. However, its sign is predictable, and in fact the coefficient of $p_\lambda$ in $\omega(\csf(P))$ is always nonnegative, where $\omega$ is the fundamental involution on symmetric functions, which sends $p_k$ to $(-1)^{k-1} p_k$. Thus, we may say that $\csf(P)$ is $\omega(p)$-positive for all $P$.

For the basis of elementary symmetric functions, the situation is more complicated. The coefficient of $e_\lambda$ in $\csf(P)$ is not positive in general, and it does not appear to have a predictable sign. However, Stanley and Stembridge~\cite{stanley95, stanley-stembridge93} have identified a large class of posets $P$ which appear to be $e$-positive. Consider the $(3+1)$ poset, which consists of the disjoint union of a chain of length 3 and a chain of length 1. This is the smallest poset which is \emph{not} $e$-positive, and as Stanley and Stembridge verified, all posets on up to 8 vertices which do not contain an induced copy of the $(3+1)$ poset have $e$-positive chromatic symmetric functions. Thus, the conjecture is that all $(3+1)$-free posets are $e$-positive.

Note that, when considering the $b$-positivity of $(3+1)$-free posets for various bases $b$, the basis of elementary symmetric functions is best-possible, in the sense that the only $e$-positive bases $b$ for which $(3+1)$-posets are $b$-positive are positive scalings of the $e$ basis. Indeed, consider the graded poset $P_\lambda$ with $\lambda_i$ vertices of rank $i$ for each $i$, and where every vertex of rank $i$ is less than every vertex of rank $i + 1$. Then, $P_\lambda$ is $(3+1)$-free, and from the definition above, we have $\csf(P_\lambda) = e_\lambda \cdot \prod_i \lambda_i!$, a scalar multiple of a single $e_\lambda$.

In the other direction, the $e$ basis is positive in the basis of Schur symmetric functions $s_\lambda$, so $e$-positivity implies $s$-positivity. Gasharov~\cite{gasharov96} proved that the chromatic symmetric functions of all $(3+1)$-free posets are $s$-positive, which gives more evidence for the $e$-positivity conjecture.

\begin{table}[b]
  \begin{tabular*}{\textwidth}{l@{\hspace{3em}}@{\extracolsep{\fill}}*{7}{r}}
    \toprule
    Number of vertices &
    1 & 2 & 3 & 4 & 5 & 6 & 7 \\
    \midrule
    All posets &
    1 & 2 & 5 & 16 & 63 & 318 & 2045 \\
    \dots$(3+1)$-free &
    1 & 2 & 5 & 15 & 49 & 173 & 639 \\
    \dots{}and $(2+2)$-free &
    1 & 2 & 5 & 14 & 42 & 132 & 429 \\
    \dots{}and basic &
    1 & 1 & 1 & 1 & 1 & 1 & 2 \\
  \end{tabular*}
  \begin{tabular*}{\textwidth}{l@{\hspace{3em}}@{\extracolsep{\fill}}*{4}{r}}
    \midrule
    Number of vertices &
    8 & 9 & 10 & 20 \\
    \midrule
    All posets &
    16999 & 183231 & 2567284 & unknown \\
    \dots$(3+1)$-free &
    2469 & 9997 & 43109 & 219364550983697100 \\
    \dots{}and $(2+2)$-free &
    1430 & 4862 & 16796 & 6564120420 \\
    \dots{}and basic &
    2 & 5 & 11 & 35635 \\
    \bottomrule
  \end{tabular*}
  \smallskip
  \caption{Numbers of posets with a given number of vertices in various classes of posets, up to isomorphism: all posets~\cite[\href{http://oeis.org/A000112}{A000112}]{sloane}; those which avoid $(3+1)$~\cite[\href{http://oeis.org/A079146}{A079146}]{sloane}; those which additionally avoid $(2+2)$~\cite[\href{http://oeis.org/A000108}{A000108}]{sloane}; those which additionally satisfy the restrictions described in \autoref{rem:subclass}.}
  \label{tab:numbers}
\end{table}


In this paper, make further progress towards the $e$-positivity conjecture, by showing that for every $(3+1)$-free poset $P$, its chromatic symmetric function $\csf(P)$ is a convex combination of the chromatic symmetric functions
\[
  \{\,\csf(P') : \text{$P'$ is $(3+1)$-free and $(2+2)$-free}\,\}.
\]
Thus, we reduce the $e$-positivity conjecture for $(3+1)$-free posets to the subclass of $(3+1)$-and-$(2+2)$-free posets, which are much more structured: these are the \emph{unit interval orders}, and they are counted by the Catalan numbers~\cite{scott-suppes58}. Note that these chromatic symmetric functions are also the subject of a recent conjecture of Shareshian and Wachs~\cite[Conjecture~5.3]{shareshian-wachs11}, which relates them to Tymoczko's~\cite{tymoczko08} representations of the symmetric groups on the equivariant cohomology of Hessenberg varieties.

In fact, our methods reduce the problem to a much smaller class of posets (see \autoref{tab:numbers} and \autoref{rem:subclass}), for which we do not currently have a satisfying characterization. However, this reduction has allowed us to computationally verify the conjecture for all posets with up to 20 vertices, using modern computer hardware, up from the previously known 8.

We also obtain a new proof of the fact that all 3-free posets (that is, posets where every vertex is either a minimal element or a maximal element) are $e$-positive~\cite[Corollary~3.6]{stanley95}.

The proofs rely on a recent structural characterization~\cite{guay-paquet-morales-rowland13} of $(3+1)$-free posets, and on a new relation which expresses the $\csf$ of a poset in terms of the $\csf$s of slightly modified versions of this poset, which we call the \emph{modular law}.

\section{Part listings}\label{sec:posets}

In order to state the modular law, we will need a representation for $(3+1)$-free posets that we can manipulate. A convenient representation for these posets is by \emph{part listings}, which we define below.

\begin{figure}[h]
  \centering
  \null
  \hfill
  \raisebox{-.5\height}{\includegraphics{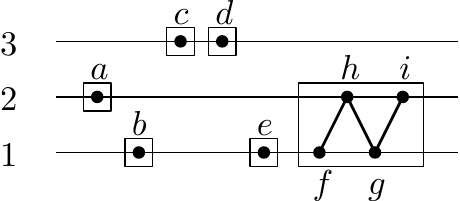}}
  \hfill
  \raisebox{-.5\height}{\includegraphics{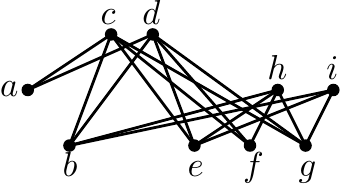}}
  \hfill
  \null
  \caption{On the left: a graphical representation of a part listing. On the right: the Hasse diagram of the corresponding poset.}
  \label{fig:part-listing}
\end{figure}

A \emph{part listing} is an ordered list of \emph{parts}, which are arranged on positive integer \emph{levels}. Each part is either a single vertex at a given level, or a bicoloured graph where the colour classes correspond to two adjacent levels and all edges join vertices on distinct levels.

\begin{example}
  \autoref{fig:part-listing} shows a part listing which consists of 6 parts spanning 3 levels, arranged left to right. The first part is the vertex $a$ on level 2; the second is vertex $b$ on level 1; the third is vertex $c$ on level 3; the fourth is vertex $d$ on level 3; the fifth is vertex $e$ on level 1; and the sixth is the bicoloured graph $G$ with vertices $\{f, g\}$ down on level 1, vertices $\{h, i\}$ up on level 2, and edges $\{fh, gh, gi\}$.
\end{example}

A part listing can be given as a word over the alphabet
\[
  \Sigma = \{\,v_i : i \in \PP\,\} \cup \{\,b_{i,i+1}(G) : \text{$i \in \PP$, $G$ a bicoloured graph}\,\},
\]
where the symbol $v_i$ corresponds to a vertex on level $i$, and the symbol $b_{i,i+1}(G)$ corresponds to a copy of the bicoloured graph $G$ on levels $i$ and $i+1$.

\begin{example}
  The part listing from \autoref{fig:part-listing} can be given as a word over the alphabet $\Sigma$ by $v_2 v_1 v_3 v_3 v_1 b_{12}(G)$, where $G$ is the bicoloured graph with vertices $\{f, g\}$ coloured `down', vertices $\{h, i\}$ coloured `up', and edges $\{fh, gh, gi\}$.
\end{example}

Given a part listing $L$, we can define an associated poset on its vertex set as follows. If $x$ and $y$ are vertices in $L$, then let $x < y$ if
\begin{enumerate}[(i),nosep]
  \item\label{item:two-levels} $x$ is at least two levels below $y$; or
  \item\label{item:one-level} $x$ is exactly one level below $y$, and the part containing $x$ appears strictly before the part containing $y$ in $L$; or
  \item\label{item:graph} $x$ is exactly one level below $y$, and they are joined by a bicoloured graph edge.
\end{enumerate}

\begin{example}
  \autoref{fig:part-listing} shows a part listing on the left and the associated poset on the right. The vertices $\{b, e, f, g\}$ are less than the vertices $\{c, d\}$ by condition~\ref*{item:two-levels}. The vertex $a$ is less than the vertices $\{c, d\}$, and the vertices $\{b, e\}$ are less than the vertices $\{h, i\}$ by condition~\ref*{item:one-level}. The relations $f < h$, $g < h$, and $g < i$ are given by condition~\ref*{item:graph}.
\end{example}

Note that this construction does yield a poset; indeed, $x < y$ implies that $x$ is on a level strictly below $y$, which guarantees anti-symmetry, and condition~\ref*{item:two-levels} guarantees transitivity. The following two propositions justify the claim that part listings are a suitable representation for posets which are $(3+1)$-free.

\begin{proposition}
  Given any part listing $L$, the associated poset $P$ is $(3+1)$-free.
\end{proposition}

\begin{proof}
  Consider a chain $x < y < z$ of three vertices in $P$, and suppose there is a vertex $w$ which is incomparable with $x$, $y$ and $z$. Vertices can only be incomparable if they are on the same or adjacent level, so $y$ and $w$ must be on the same level, with $x$ on the level below and $z$ on the level above. Given the relations in $P$ between these vertices, it can be seen that $y$ cannot appear before $w$ in $L$ because of $x$; that $y$ cannot appear after $w$ because of $z$; and that $y$ cannot be in the same part as $w$, as this part would also contain $x$ and $z$, and span more than two levels. Thus, $P$ cannot contain an induced copy of the $(3+1)$ poset.
\end{proof}

\begin{proposition}
  Given any $(3+1)$-free poset $P$, there exists a part listing $L$ for which the associated poset is $P$.
\end{proposition}

\begin{proof}
  By~\cite[Theorem~3.3]{guay-paquet-morales-rowland13}, every $(3+1)$-free poset $P$ has a \emph{compatible listing} $L'$ made up of \emph{clone sets} and \emph{tangles}. A clone set $c_i$ with $k$ vertices at level $i$ in the compatible listing $L'$ corresponds to $k$ consecutive parts $v_i$ in the part listing $L$, and a tangle $t_{i,i+1}(G)$ in $L'$ corresponds to a bicoloured graph $b_{i,i+1}(G)$ in $L$.
\end{proof}

Note that two different part listings $L$ and $L'$ can give rise to the same poset, in which case we say that they are \emph{poset-equivalent}, and write $L \peq L'$. In particular, this happens when $L'$ is obtained from $L$ by applying a sequence of commutation, circulation and/or combination relations, as described below.

\subsection{Commutation relations}

If two consecutive parts of a part listing $L$ are at least two levels apart, then they can safely be swapped without interfering with the definition of the associated poset. That is, if $A$, $B$ are words over the alphabet $\Sigma$ and $i$, $j$ are levels with $j - i \geq 2$, then we have the relations
\begin{align*}
  A v_i v_j B &\peq A v_j v_i B, &
  A v_i b_{j,j+1} B &\peq A b_{j,j+1} v_i B, \\
  A b_{i,i+1} v_{j+1} B &\peq A v_{j+1} b_{i,i+1} B, &
  A b_{i,i+1} b_{j+1,j+2} B &\peq A b_{j+1,j+2} b_{i,i+1} B.
\end{align*}

\begin{example}
  For the part listing given in \autoref{fig:part-listing}, the second and third parts, corresponding to vertices $b$ and $c$, can be swapped without changing the associated poset, so $v_2 v_1 v_3 v_3 v_1 b_{12}(G) \peq v_2 v_3 v_1 v_3 v_1 b_{12}(G)$.
\end{example}

\subsection{Circulation relations}

Given a word $A$ over the alphabet $\Sigma$, let $A^+$ be the word obtained by raising each symbol by one level, that is, replacing each $v_i$ by $v_{i+1}$ and each $b_{i,i+1}(G)$ by $b_{i+1,i+2}$. Then, it can be checked that, for any two words $A$, $B$ over $\Sigma$, we have the poset-equivalence
\[
  A^+ B \peq B A.
\]
In particular, if a part listing $L$ starts with a part on level 2 or above, then this first part can be lowered by one level and moved to the end of $L$ without changing the associated poset.

\begin{example}
  For the part listing given in \autoref{fig:part-listing}, the first part, corresponding to vertex $a$, is on level 2, so it can be lowered to level 1 and moved to the end without changing the associated poset. Thus, $v_2 v_1 v_3 v_3 v_1 b_{12}(G) \peq v_1 v_3 v_3 v_1 b_{12}(G) v_1$.
\end{example}

\subsection{Combination relations}

If two or more consecutive parts in a part listing all lie on levels $i$ and $i+1$, then they can be replaced by a single equivalent bicoloured graph part, where the edges of the graph are given by the poset relations between the vertices involved. Conversely, it may be possible to decompose a single bicoloured graph part into a sequence of consecutive parts occupying the same two levels. If $B_{i,i+1}$ is a word over the alphabet
\[
  \{v_i, v_{i+1}\} \cup \{\,b_{i,i+1}(G) : \text{$G$ a bicoloured graph}\,\},
\]
let us write $\overline{B_{i,i+1}}$ for the equivalent bicoloured graph part. Then, for any two words $A$, $C$ over $\Sigma$, we have the poset-equivalence relation
\[
  A B_{i,i+1} C \peq A \overline{B_{i,i+1}} C.
\]

\begin{example}
  For the part listing given in \autoref{fig:part-listing}, the last two parts, corresponding to vertex $e$ on level 1 and the vertices $\{f, g, h, i\}$ on levels 1 and 2, can be combined without changing the associated poset. Thus, $v_2 v_1 v_3 v_3 v_1 b_{12}(G) \peq v_2 v_1 v_3 v_3 \overline{v_1 b_{12}(G)} = v_2 v_1 v_3 v_3 b_{12}(G')$, where $G'$ is the bicoloured graph with vertices $\{e, f, g\}$ coloured `down', vertices $\{h, i\}$ coloured `up', and edges $\{eh, ei, fh, gh, gi\}$.
\end{example}

Note that the \emph{tangles} defined in~\cite{guay-paquet-morales-rowland13} are exactly the bicoloured graphs which cannot be decomposed using combination relations.

\section{The modular law}\label{sec:modularity}

With the notation of part listings for $(3+1)$-free posets in place, we can now state the modular law for their chromatic symmetric functions.

\begin{proposition}[modular law]\label{prop:modular}
  Consider the part listing $A b_{i,i+1}(G) B$, where~$A$,~$B$ are words over $\Sigma$ and $G$ is a bicoloured graph. Suppose $G$ contains two edges~$e_1$,~$e_2$ incident to a common vertex $y$, so that $e_1 = xy$ and $e_2 = yz$ for some vertices~$x$,~$z$. Let $G_1$, $G_2$, and $G_{12}$ be the graphs obtained from $G$ by removing the edge $e_1$, the edge $e_2$, and both edges, respectively (but no vertices). Let $P$, $P_1$, $P_2$, and $P_{12}$ be the posets associated to the part listings $A b_{i,i+1}(G) B$, $A b_{i,i+1}(G_1) B$, $A b_{i,i+1}(G_2) B$, and $A b_{i,i+1}(G_{12}) B$, respectively. Then,
  \[
    \csf(P) + \csf(P_{12}) = \csf(P_1) + \csf(P_2).
  \]
\end{proposition}

\begin{proof}
  Since the chromatic symmetric function is a generating function for proper vertex colourings and all of these posets have the same vertex set, it is enough to verify that each vertex colouring $\kappa \colon V \to \PP$ (whether proper or not) makes the same contribution to both sides of the equation. For $\kappa$ to make any contribution, it must be proper for $P$. In particular, $\kappa(x) \neq \kappa(z)$, and $\kappa(y)$ may be equal to $\kappa(x)$, or $\kappa(z)$, or neither. If $\kappa(y) = \kappa(x)$, then $\kappa$ is proper for $P$ and $P_2$, but not for $P_1$ nor $P_{12}$. If $\kappa(y) = \kappa(z)$, then $\kappa$ is proper for $P$ and $P_1$, but not for $P_2$ nor $P_{12}$. If $\kappa(y)$ is distinct from $\kappa(x)$ and $\kappa(z)$, then $\kappa$ is proper for all of $P$, $P_1$, $P_2$ and $P_{12}$. In all cases, the contribution of $\kappa$ is the same to both sides of the equation.
\end{proof}

Note that the proof of the modular relation only depends on the fact that $x < y$ and $z < y$ (or, symmetrically, $y < x$ and $y < z$) are cover relations in the poset $P$. However, we state the modular law as above so that all the posets involved are manifestly $(3+1)$-free.

Since we are concerned with linear combinations of chromatic symmetric functions and linear relations between them, it will be useful to consider formal linear combinations of part listings, and to define the $\csf$ and poset-equivalence $\peq$ on these linear combinations by linear extension. Then, we define a modular-equivalence relation $\meq$ by imposing
\[
  L + L_{12} \meq L_1 + L_2,
\]
where $L$, $L_1$, $L_2$, and $L_{12}$ are the part listings considered in \autoref{prop:modular}, and extending $\meq$ linearly so that it is invariant under translation and scaling. Then, we have $\csf(\alpha) = \csf(\beta)$ whenever $\alpha \peq \beta$ or $\alpha \meq \beta$.

\section{Dual bases for bicoloured parts}\label{sec:bases}


Now, let us consider some consequences of the modular law for the computation of chromatic symmetric functions. In particular, let us fix a set of $r$ vertices coloured `down' and a set of $s$ vertices coloured `up' and look at the modular law when restricted to bicoloured graphs on these vertices. More formally, let $V_r^s$ be the vector space over $\QQ$ of formal linear combinations of the set
\[
  \{\,b_{12}(G) : \text{$G$ is a bicoloured graph with $r$ vertices below and $s$ vertices above}\,\},
\]
modulo the modular-equivalence relation $\meq$. Also, consider the vectors
\[
  U_k = \overline{v_2^{s-k} v_1^r v_2^k} \in V_r^s \qquad \text{for $k = 0, 1, 2, \ldots, s$,}
\]
which we call \emph{udu vectors} (for `up-down-up'), and the vectors
\[
  D_k = \overline{v_1^k v_2^s v_1^{r-k}} \in V_r^s \qquad \text{for $k = 0, 1, 2, \ldots, r$,}
\]
which we call \emph{dud vectors}. Finally, consider the linear functionals $F_k \colon V_r^s \to \QQ$ for $k = 0, 1, 2, \ldots, \min\{r, s\}$ defined as follows, which we call \emph{probability functionals}:
\begin{quote}
  Let $G$ be a bicoloured graph with $r$ vertices below and $s$ vertices above. Let $M$ be a random matching with $\min\{r, s\}$ edges from the complete bicoloured graph on the same vertex set, taken uniformly at random out of all $\max\{r, s\}! / \abs{r - s}!$ such matchings. Then, $F_k(b_{12}(G))$ is the probability that $G$ and $M$ have exactly $k$ edges in common.
\end{quote}
Note that this is well-defined, since the vectors $b_{12}(G)$ span $V_r^s$, and $F_k$ respects the modular-equivalence relation
\[
  b_{12}(G) + b_{12}(G_{12}) \meq b_{12}(G_1) + b_{12}(G_2).
\]
The following proposition summarizes some useful properties of the probability functionals, the udu vectors, and the dud vectors.

\begin{proposition}[dual bases]\label{prop:dual}
  \hspace*{0pt} 
  \begin{enumerate}[(i),nosep]
    \item
      The probability functionals are a basis for space of linear functionals on $V_r^s$.
    \item
      If $r \geq s$, then the udu vectors are a dual basis for the probability functionals, in the sense that $F_j(U_k) = 1$ is $j = k$ and $0$ otherwise.
    \item
      If $r \leq s$, then the dud vectors are a dual basis for the probability functionals, in the sense that $F_j(D_k) = 1$ is $j = k$ and $0$ otherwise.
    \item
      Every vector of the form $b_{12}(G)$ in $V_r^s$ is a convex combination of udu vectors or dud vectors.
  \end{enumerate}
\end{proposition}

\begin{proof}
  \begin{enumerate}[(i),nosep]
    \item
      For $k = 0, 1, 2, \ldots, \min\{r, s\}$, let $M_k$ be the (unique, up to isomorphism) bicoloured matching with $r$ vertices below, $s$ vertices above, and $k$ edges. Consider the vector $b_{12}(G) \in V_r^s$ for an arbitrary $G$. If $G$ has a vertex of degree more than 1, let $e_1$ and $e_2$ be two edges incident to that vertex. Then, the modular relation gives
      \[
        b_{12}(G) \meq b_{12}(G_1) + b_{12}(G_2) - b_{12}(G_{12}),
      \]
      and by induction on the number of edges in $G$, it follows that $b_{12}(G)$ can be expressed as
      \[
        b_{12}(G) \meq \sum_{k=0}^{\min\{r, s\}} c_k b_{12}(M_k)
      \]
      for some coefficients $c_k$. Thus, the vectors $b_{12}(M_k)$ span $V_r^s$. Also, the probability $F_j(b_{12}(M_k))$ is nonzero if $j = k$, and zero if $j > k$, so it follows that the vectors $b_{12}(M_k)$ are linearly independent. Symmetrically, the linear functionals $F_j$ are linearly independent, and by a dimension argument, they form a basis for space of linear functionals on $V_r^s$.

    \item
      Direct computation.

    \item
      Direct computation.

    \item
      The numbers $F_k(b_{12}(G))$ give the coefficients of $b_{12}(G)$ in the basis of udu vectors or dud vectors, depending on whether $r \geq s$ or $r \leq s$. Since these numbers are the probabilities of a set of events which partition their sample space, they are nonnegative and their sum is 1.
      \qedhere
  \end{enumerate}
\end{proof}

\section{Consequences}\label{sec:proofs}

\begin{theorem}
  If every $(3+1)$- and $(2+2)$-free poset is $e$-positive, then every $(3+1)$-free poset is $e$-positive.
\end{theorem}

\begin{proof}
  Every $(3+1)$-free poset can be represented as a part listing, possibly containing parts of the form $b_{i,i+1}(G)$. By \autoref{prop:dual}, each part of the form $b_{i,i+1}(G)$ can be replaced by a convex combination of udu vectors or dud vectors by using modular-equivalence without affecting the chromatic symmetric function. Furthermore, each udu vector or dud vector is poset-equivalent to a list of parts with no bicoloured graphs. By~\cite{guay-paquet-morales-rowland13}, the posets associated to part listings with no parts of the form $b_{i,i+1}(G)$ are exactly the $(3+1)$-and-$(2+2)$-free posets.
\end{proof}

\begin{remark}\label{rem:subclass}
  In fact, the question of $e$-positivity for $(3+1)$-free posets can be further reduced to a much smaller class of possible counter-examples. A first reduction can be obtained by throwing out not only part listings which contain parts of the form $b_{i,i+1}(G)$ other than udu vectors or dud vectors, but also those which contain parts of this form after applying the relations of poset-equivalence.

  A second reduction can be obtained by noting that, if the vertex set of a poset $P$ can be split into nonempty sets $X$ and $Y$ such that $x < y$ for all $x \in X$ and $y \in Y$, then $\csf(P)$ can be computed as the product of the chromatic symmetric functions of $P$ restricted to $X$ and $Y$. Thus, only posets which cannot be split in this way need to be checked.

  A third reduction can be obtained by using the fact that the udu vector $U_k \in V_r^s$ and the dud vector $D_k \in V_r^s$ are equal up to modular-equivalence when $r = s$, combined with the first or second reductions.

  A fourth reduction can be obtained that the chromatic symmetric function is invariant under reversing all the relations of a poset.

  After performing all these reductions, a comparatively tiny class is posets remains (see \autoref{tab:numbers}). In particular, there are only 62152 such posets on up to 20 vertices, and we have computationally checked that they are all $e$-positive.
\end{remark}

\begin{theorem}
  Every $3$-free poset is $e$-positive.
\end{theorem}

\begin{proof}
  Let $P$ be a $3$-free poset, so that it can be represented by the part listing $b_{12}(G)$ for some bicoloured graph $G$ with $r$ vertices on level 1 and $s$ vertices on level 2. By turning $P$ upside down if necessary, we can assume that $r \geq s$. Then, up to modular-equivalence, we can express $b_{12}(G)$ as the convex combination
  \[
    b_{12}(G) \meq
      \sum_{k = 0}^s c_k \cdot \overline{v_2^{s-k} v_1^r v_2^k}.
  \]
  For the term $k = s$ appearing on the right-hand side, we can directly compute
  \[
    \csf(\overline{v_1^r v_2^s}) = r! \, s! \cdot e_{r, s}.
  \]
  For each remaining term, we can use the circulation relation to transform the leading $v_2$ into a trailing $v_1$, rewriting the term as
  \[
    \overline{v_2^{s-k} v_1^r v_2^k} \peq \overline{v_2^{s-k-1} v_1^r v_2^k v_1} \in V_{r+1}^{s-1}.
  \]
  By repeating this process of expressing as a convex combination, peeling off a term, and rewriting the remaining terms, we eventually obtain
  \[
    \csf(b_{12}(G))
      = \sum_{k = 0}^s c'_k \csf(\overline{v_1^{r+k} v_2^{s-k}})
      = \sum_{k = 0}^s c'_k (r+k)! \, (s-k)! \cdot e_{r+k,s-k},
  \]
  where the $c'_k$ are the coefficients of a convex combination.
\end{proof}

\pagebreak

\begin{example}
  We can compute
  \[
    \graphic{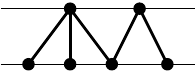}
      \meq \tfrac{5}{12} \graphic{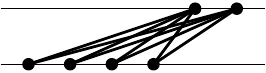}
      + \tfrac{5}{12} \graphic{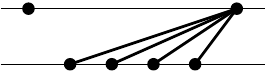}
      + \tfrac{2}{12} \graphic{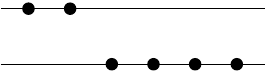},
  \]
  then
  \begin{align*}
    \tfrac{5}{12} \graphic{example-rhs-141}
      + \tfrac{2}{12} \graphic{example-rhs-24}
      &\peq \tfrac{5}{12} \graphic{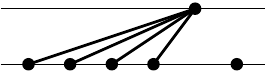}
      + \tfrac{2}{12} \graphic{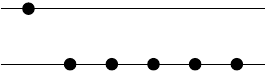} \\
      &\meq \tfrac{20}{60} \graphic{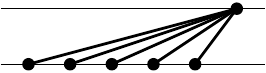}
      + \tfrac{15}{60} \graphic{example-rhs-15}, \\
  \end{align*}
  then
  \[
    \tfrac{15}{60} \graphic{example-rhs-15}
      \peq \tfrac{15}{60} \graphic{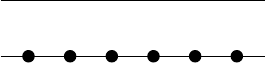},
  \]
  and we know that
  \begin{align*}
    \csf\graphic{example-rhs-042} &= 4! \, 2! \cdot e_{42}, \\
    \csf\graphic{example-rhs-051} &= 5! \, 1! \cdot e_{51}, \\
    \csf\graphic{example-rhs-06} &= 6! \cdot e_{6},
  \end{align*}
  so we have
  \[
    \csf\graphic{example-lhs} = 20 e_{42} + 40 e_{51} + 180 e_{6}.
  \]
\end{example}

\section{Acknowledgements}\label{sec:thanks}

This work grew out of a working session of the algebraic combinatorics group at LaCIM with active participation from Chris Berg, Alejandro Morales, Eric Rowland, Franco Saliola, and Luis Serrano. It was facilitated by computer exploration using various mathematical software packages, including Sage~\cite{sage}, its Sage-Combinat extensions~\cite{sage-combinat} and the nauty suite of programs~\cite{nauty}, and hardware provided by Franco Saliola and funded by the FRQNT through its ``Établissement de nouveaux chercheurs universitaires'' program. The author would also like to thank Philippe Nadeau for helpful discussions.

\bibliographystyle{plainurl}
\bibliography{references}


\end{document}